\documentclass[11pt]{amsart}
\usepackage{pdfsync}
\usepackage{latexsym}
\usepackage{amsgen,amsmath,amsxtra,amssymb,amsfonts,amsthm}
\usepackage{times}
\usepackage{graphicx}
\usepackage[dvips]{epsfig}
\usepackage{subfigure}
\usepackage[active]{srcltx}
\usepackage{color}
\usepackage[makeroom]{cancel}

\begin{document}
\newcommand{\dyle}{\displaystyle}
\newcommand{\R}{{\mathbb{R}}}
 \newcommand{\Hi}{{\mathbb H}}
\newcommand{\Ss}{{\mathbb S}}
\newcommand{\N}{{\mathbb N}}
\newcommand{\Rn}{{\mathbb{R}^n}}
\newcommand{\ieq}{\begin{equation}}
\newcommand{\eeq}{\end{equation}}
\newcommand{\ieqa}{\begin{eqnarray}}
\newcommand{\eeqa}{\end{eqnarray}}
\newcommand{\ieqas}{\begin{eqnarray*}}
\newcommand{\eeqas}{\end{eqnarray*}}
\newcommand{\Bo}{\put(260,0){\rule{2mm}{2mm}}\\}
\def\L#1{\label{#1}} \def\R#1{{\rm (\ref{#1})}}


\theoremstyle{plain}
\newtheorem{theorem}{Theorem} [section]
\newtheorem{corollary}[theorem]{Corollary}
\newtheorem{lemma}[theorem]{Lemma}
\newtheorem{proposition}[theorem]{Proposition}
\def\neweq#1{\begin{equation}\label{#1}}
\def\endeq{\end{equation}}
\def\eq#1{(\ref{#1})}

\theoremstyle{definition}
\newtheorem{definition}[theorem]{Definition}
\newtheorem{remark}[theorem]{Remark}

\numberwithin{figure}{section}
\newcommand{\res}{\mathop{\hbox{\vrule height 7pt width .5pt depth
0pt \vrule height .5pt width 6pt depth 0pt}}\nolimits}
\def\at#1{{\bf #1}: } \def\att#1#2{{\bf #1}, {\bf #2}: }
\def\attt#1#2#3{{\bf #1}, {\bf #2}, {\bf #3}: } \def\atttt#1#2#3#4{{\bf #1}, {\bf #2}, {\bf #3},{\bf #4}: }
\def\aug#1#2{\frac{\displaystyle #1}{\displaystyle #2}} \def\figura#1#2{ \begin{figure}[ht] \vspace{#1} \caption{#2}
\end{figure}} \def\B#1{\bibitem{#1}} \def\q{\int_{\Omega^\sharp}}
\def\z{\int_{B_{\bar{\rho}}}\underline{\nu}\nabla (w+K_{c})\cdot
\nabla h} \def\a{\int_{B_{\bar{\rho}}}}
\def\b{\cdot\aug{x}{\|x\|}}
\def\n{\underline{\nu}} \def\d{\int_{B_{r}}}
\def\e{\int_{B_{\rho_{j}}}} \def\LL{{\mathcal L}}
\def\itr{\mathrm{Int}\,}
\def\D{{\mathcal D}}
 \def\tg{\tilde{g}}
\def\A{{\mathcal A}}
\def\S{{\mathcal S}}
\def\H{{\mathcal H}}
\def\M{{\mathcal M}}
\def\T{{\mathcal T}}
\def\U{{\mathcal U}}
\def\N{{\mathcal N}}
\def\I{{\mathcal I}}
\def\F{{\mathcal F}}
\def\J{{\mathcal J}}
\def\E{{\mathcal E}}
\def\F{{\mathcal F}}
\def\G{{\mathcal G}}
\def\HH{{\mathcal H}}
\def\W{{\mathcal W}}
\def\H{\D^{2*}_{X}}
\def\d{d^X_M }
\def\LL{{\mathcal L}}
\def\H{{\mathcal H}}
\def\HH{{\mathcal H}}
\def\itr{\mathrm{Int}\,}
\def\vah{\mbox{var}_\Hi}
\def\vahh{\mbox{var}_\Hi^1}
\def\vax{\mbox{var}_X^1}
\def\va{\mbox{var}}
\def\SS{{\mathcal S}}
 \def\Y{{\mathcal Y}}
\def\length{{l_\Hi}}
\newcommand{\average}{{\mathchoice {\kern1ex\vcenter{\hrule
height.4pt width 6pt depth0pt} \kern-11pt} {\kern1ex\vcenter{\hrule height.4pt width 4.3pt depth0pt} \kern-7pt} {} {} }}
\def\weak{\rightharpoonup}
\def\detu{{\rm det}(D^2u)}
\def\detut{{\rm det}(D^2u(t))}
\def\detvt{{\rm det}(D^2v(t))}
\def\detv{{\rm det}(D^2v)}
\def\uuu{u_xu_yu_{xy}}
\def\uuut{u_x(t)u_y(t)u_{xy}(t)}
\def\uuus{u_x(s)u_y(s)u_{xy}(s)}
\def\uuutn{u_x(t_n)u_y(t_n)u_{xy}(t_n)}
\def\vvv{v_xv_yv_{xy}}
\newcommand{\ave}{\average\int}

\title[Explicit blowing up solutions]{Explicit blowing up solutions for a higher
order parabolic equation with Hessian nonlinearity}

\author[C. Escudero]{Carlos Escudero}
\address{}
\email{}

\keywords{Fourth order partial differential equations, Hessian nonlinearity, finite time blow-up,
infinite time blow-up, complete blow-up, explicit solutions.
\\ \indent 2010 {\it MSC: 35A01, 35B44, 35C05, 35G25, 35G31, 35Q82.}}

\date{\today}

\begin{abstract}
In this work we consider a nonlinear parabolic higher order partial differential equation that has been proposed as a model for epitaxial growth.
This equation possesses both global-in-time solutions and solutions that blow up in finite time, for which this blow-up is mediated by its Hessian
nonlinearity.
Herein, we further analyze its blow-up behaviour by means of the construction of explicit solutions in the square, the disc, and the plane. Some of
these solutions show complete blow-up in either finite or infinite time. Finally, we refine a blow-up criterium that was proved for this evolution equation. Still, existent blow-up criteria based on {\it a priori} estimates do not completely reflect the singular character of these explicit blowing up solutions.
\end{abstract}

\maketitle

\section{Introduction}

In this work we consider both the initial and initial-boundary value problems for the partial differential equation
\begin{equation}
\label{pareq}
u_t= \det(D^2 u) -\Delta^2 u,
\end{equation}
posed either on $\mathbb{R}^2$ or on a bounded subset of the plane.
Higher order equations have attracted attention for both
its mathematical structure, as they in general lack the maximum principles characteristic of second order models, as well as their interest
in different applications~\cite{GGS}. Higher order equations constitute one of the current trends in the modern theory of partial differential
equations, such as equations that involve fractional derivatives~\cite{fksm}.

Equation~\eqref{pareq} is a higher order parabolic equation provided with a Hessian nonlinearity and, because of that,
it possesses no second order analogue; in this respect, it is an interesting model to be analyzed. It has been studied in the context of
condensed matter physics as a model for epitaxial growth~\cite{escudero,ek}. Its stationary solutions have been analyzed in~\cite{n3,n4,n5}.
Its self-similar solutions were studied in~\cite{n1}. Its higher dimensional counterparts have been considered in~\cite{n0,escudero2,n6}.
Its numerical analysis was carried out in~\cite{vpe}.
And the initial-boundary value problem for this equation, posed on a bounded domain with either Dirichlet or Navier boundary conditions,
was initially analyzed in~\cite{n2}, and the analysis was subsequently developed in~\cite{xu,xuzhou,zhou,zhou2}.

To connect equation~\eqref{pareq} to epitaxial growth we consider the function
\begin{equation}
\begin{aligned}\nonumber
&\sigma : \Omega \subseteq \mathbb{R}^{N}\times \mathbb{R}_{+} \rightarrow \mathbb{R},
\end{aligned}
\end{equation}
which mathematically describes the height of the epitaxially growing interface in the spatial point $(x,y) \in \Omega \subseteq \mathbb{R}^{2} $ at the time instant $t \in \mathbb{R}_{+}$. We reduce the study to the case $N=2$, since it is the one that corresponds to the actual physical situation. The macroscopic description of the growing interface is given by a suitable partial differential equation to be solved for $\sigma$. There are several
partial differential equations of this sort, see for instance the discussion in~\cite{escudero} and~\cite{ek}, and references therein.
We now summarize the geometric derivation of this model following~\cite{n3}.
Let us start assuming that the height function $\sigma$ obeys a gradient flow with a forcing term
\begin{equation}\nonumber
\sigma_t=\left(1+(\nabla \sigma)^{2}\right)^{\frac{1}{2}}\left[-\frac{\delta J(\sigma)}{\delta \sigma}+\zeta(x,y,t)\right].
\end{equation}
The functional $J(\sigma)$ denotes a potential that describes the microscopic properties of the interface at the macroscopic scale and is given by
\begin{equation}\nonumber
J(\sigma)=\int_{\Omega} Q(z)\left(1+(\nabla \sigma)^{2}\right)^{\frac{1}{2}} ~dx~dy,
\end{equation}
where the square root term models growth along the normal to the
surface (i.~e. the mathematical expression of the physical interface),
$z$ denotes the mean curvature of the graph of $\sigma(x,y,\cdot)$, which analytically describes the surface,
and $Q$ is an unknown function of $z$.
After assuming analyticity on $Q(z)$, it can be expressed as its power series expansion given by
\begin{equation}\nonumber
Q(z)=K_{0}+K_{1}z+K_{2} \frac{z^{2}}{2!}+K_{3}\frac{z^{3}}{3!}+\cdots.
\end{equation}
Applying the small gradient expansion, which assumes $|\nabla \sigma| \ll 1$, together with this power series expansion, yields, after simplification, the reduced expression
\begin{equation}\nonumber
\sigma_t=K_{0} \, \Delta \sigma+2 K_{1}~\det(D^{2} \sigma)~-~K_{2} \, \Delta^{2} \sigma-\frac{1}{2} K_{3} \, \Delta(\Delta\sigma)^{2}+\zeta(x,y,t),
\end{equation}
for the equation, if only linear and quadratic terms are retained.
For the relation between this equation and other models in this field, along with a geometrical explanation of each term in it, we refer the reader
to~\cite{n3}. For modeling reasons we can assume that $K_{0}=0~\mbox{and}~K_{3}=0$, see again~\cite{n3}. Therefore it reduces to
\begin{equation}\nonumber
\sigma_t + K_{2} \, \Delta^{2} \sigma=2 K_{1}~\mbox{det}(D^{2} \sigma)+\zeta(x,y,t).
\end{equation}
This equation can be considered as a possible continuum description of epitaxial growth~\cite{escudero,ek}.
Equation~\eqref{pareq} arises after non-dimensionalization and by suppressing the time-dependent forcing term (i.~e. we set $\zeta(x,y,t) \equiv 0$),
which is a measure of the speed at which new particles are deposited into the system. In this autonomous case we denote the solution by $u$ rather
than $\sigma$ (to be precise, they are different, since $u$ is the non-dimensionalization of $\sigma$).
Note that the non-autonomous case is also accessible to mathematical analysis~\cite{n2}, but in the search of explicit solutions, like
herein, it is more convenient to consider the autonomous situation.

Along this work we complement the previous results on blowing up solutions for the evolution problem that have been found in~\cite{n2,xu,xuzhou,zhou,zhou2}.
Our approach is two-fold: first, we compute explicit solutions to model~\eqref{pareq} that belong to three different classes;
some blow up in finite time, some blow up in infinite time, and others do not blow up at all. This program is carried out in section~\ref{explicit} by
means of three complementary problems: one is posed on the unit square, other is posed on the unit disc, and the third one is posed
on the whole of the plane.
In all the cases where it happens, the blow-up is shown to be complete; and some regularities in the blow-up structure of all of these examples are appreciated. The blow-ups
are then so wild that they exceed the estimations shown in~\cite{n2,xu,xuzhou,zhou,zhou2}. For instance, in~\cite{n2}, the blow-up is shown to be in the
Sobolev space $H^2(\Omega)$, but this in turn implies the divergence of the $W^{1,\infty}(\Omega)-$norm of the solution; this fact is proven in our current section~\ref{refine},
and this section indeed constitutes the second step of our approach. But also note that the blow-up is shown to be in $L^2(\Omega)$ in~\cite{zhou}.
Still, our explicit examples show a more dramatic behavior than the one expected from the {\it a priori} estimates.
Moreover, under the precise conditions assumed in~\cite{n2}, infinite time blow-up is proven not to happen.
But in all the situations considered herein, at least in the cases exhaustively analyzed,
solutions that blow up in infinite time are present. Overall, this shows that equation~\eqref{pareq} still presents some features that are not
completely captured by the theoretical developments built so far.
Our main conclusions, which highlight some open questions, are drawn in section~\ref{conclusions}.

\section{Explicit formulas}
\label{explicit}

\subsection{Explicit solutions on the unit square}

Consider equation~\eqref{pareq}
in the domain $[0,1]^2$ subject to the boundary conditions
\begin{eqnarray}\nonumber
& & \partial_x^3 u(0,y,t)= \partial_x^3 u(1,y,t)= \partial_y^3 u(x,0,t)= \partial_y^3 u(x,1,t)=0, \\ \nonumber
& & \partial_x u(0,y,t)= \partial_y u(x,0,t)= 0, \\ \nonumber
& & \partial_x u(1,y,t) = \partial_x^2 u(1,y,t), \\ \nonumber
& & \partial_y u(x,1,t) = \partial_y^2 u(x,1,t).
\end{eqnarray}
It admits the explicit solution
$$
u(x,y,t)= \frac{a_0}{1 + 12 \, a_0 \, t} \, x^2 \, y^2 - \frac{2}{3} \ln(1 + 12 \, a_0 \, t).
$$
For $a_0=0$ the solution $u \equiv 0$ for all times. For $a_0 > 0$ the solution exists for all times but blows up in infinite time
in such a way that
$$
\lim_{t \nearrow \infty} u(x,y,t)= - \infty,
$$
for all $(x,y) \in [0,1]^2$.
For
$a_0 < 0$ the solution blows up in finite time:
$$
T^* = \frac{-1}{12 a_0}.
$$
In this case we observe
$$
\lim_{t \nearrow T^*} u(x,y,t)= + \infty,
$$
in the sets $[0,1]^2 \cap \{x =0\}$ and $[0,1]^2 \cap \{y=0\}$. We also find
$$
\lim_{t \nearrow T^*} u(x,y,t)= - \infty,
$$
for all other $(x,y) \in [0,1]^2$.
We note that in all cases the blow-up is complete.

\subsection{Explicit solutions on the unit disc}

On the unit disc our model reads
$$
u_t = \frac{u_r \, u_{rr}}{r} - \Delta^2_r u,
$$
where $r = \sqrt{x^2 + y^2}$ and $\Delta_r(\cdot)=\frac{1}{r}[r(\cdot)_r]_r$.
We consider it now subject to the boundary conditions
\begin{eqnarray}\nonumber
& & 3 \, \partial_r u(1,t) = \partial_r^2 u(1,t), \\ \nonumber
& & 2 \, \partial_r^2 u(1,t) = \partial_r^3 u(1,t), \\ \nonumber
& & \partial_r u(0,t) = \partial_r^3 u(0,t)= 0.
\end{eqnarray}
The explicit solution in this case reads
$$
u(r,t)= \frac{a_0}{1 - 48 \, a_0 \, t} \, r^4 + \frac{4}{3} \ln(1 - 48 \, a_0 \, t).
$$
For $a_0=0$ the solution $u \equiv 0$ for all times. For $a_0 < 0$ the solution exists for all times but blows up in infinite time
in this fashion:
$$
\lim_{t \nearrow \infty} u(r,t)= +\infty,
$$
for all $r \in [0,1]$. For
$a_0 > 0$ the solution blows up in finite time:
$$
T^* = \frac{1}{48 a_0},
$$
and it fulfils
$$
\lim_{t \nearrow T^*} u(0,t)= -\infty,
$$
and
$$
\lim_{t \nearrow T^*} u(r,t)= +\infty,
$$
for all $r \in \, ]0,1]$.

\subsection{Families of solutions on the plane}

On $\mathbb{R}^2$, and in the non-radial case, we have the following families of solutions
\begin{eqnarray}\nonumber
u(x,y,t) &=& \frac{a_0}{1 + 12 \, a_0 \, a_3 \, t} \, \left( \frac{3456 \, a_1^3 \, a_3^3 \, + 432 \, a_1^2 \, a_3^2 -1}{46656 \, a_2^2 \, a_3^3} \, x^4 \right.
+ a_1 \, x^2 \, y^2 \\ \nonumber
& & + a_2 \, x \, y^3 + \frac{288 \, a_1^3 \, a_3^2 + 12 \, a_1 \, a_3 -1}{648 \, a_2 \, a_3^2} \, x^3 \, y + \left. \frac{9 \, a_2^4 \, a_3}{24 \, a_1 \, a_3 -1} \, y^4 \right) \\ \nonumber
& & - \frac{1}{12 \, a_3} \left( 4 \, a_1 +
\frac{3456 \, a_1^3 \, a_3^3 + 432 \, a_1^2 \, a_3^2 -1}{1944 \, a_2^2 \, a_3^3} + \frac{216 \, a_2^2 \, a_3}{24 \, a_1 \, a_3 -1} \right) \\ \label{solr21}
& & \times \ln(1 + 12 \, a_0 \, a_3 \, t) + a_4 \, x + a_5 \, y + a_6,
\end{eqnarray}
where $a_0$, $a_1$, $a_2$, $a_3$, $a_4$, $a_5$ and $a_6$ are arbitrary real parameters expect for:
\begin{eqnarray}\nonumber
a_2 &\neq& 0, \\ \nonumber
a_3 &\neq& 0, \\ \nonumber
a_1 a_3 &\neq& 1/24.
\end{eqnarray}
The solution blows up in finite time
$$
T^* = \frac{-1}{12 \, a_0 \, a_3}
$$
provided $a_0 \, a_3 < 0$, in which case it blows up on the set
\begin{eqnarray}\label{bumanifold}
&& \left\{ (x,y) \in \mathbb{R}^2 \left| \,\,
\frac{3456 \, a_1^3 \, a_3^3 \, + 432 \, a_1^2 \, a_3^2 -1}{46656 \, a_2^2 \, a_3^3} \, x^4 \right. \right.
+ a_1 \, x^2 \, y^2 + a_2 \, x \, y^3 \\ \nonumber && \hspace{2.25cm}
+ \frac{288 \, a_1^3 \, a_3^2 + 12 \, a_1 \, a_3 -1}{648 \, a_2 \, a_3^2} \, x^3 \, y + \left. \frac{9 \, a_2^4 \, a_3}{24 \, a_1 \, a_3 -1} \, y^4
\neq 0 \right\},
\end{eqnarray}
unless the condition
\begin{equation}\label{algebraic}
4 \, a_1 +
\frac{3456 \, a_1^3 \, a_3^3 + 432 \, a_1^2 \, a_3^2 -1}{1944 \, a_2^2 \, a_3^3} + \frac{216 \, a_2^2 \, a_3}{24 \, a_1 \, a_3 -1}
\neq 0
\end{equation}
holds, in which case the blow-up is complete. Note that condition~\eqref{algebraic} is immediately met whenever $a_1=0$.
In such a case formula~\eqref{solr21} reduces to
\begin{eqnarray}\nonumber
u(x,y,t) &=& -\frac{a_0}{1 + 12 \, a_0 \, a_3 \, t} \\ \nonumber &&
\times \left( \frac{1}{46656 \, a_2^2 \, a_3^3} \, x^4
- a_2 \, x \, y^3 + \frac{1}{648 \, a_2 \, a_3^2} \, x^3 \, y + 9 \, a_2^4 \, a_3 \, y^4 \right) \\ \nonumber
& & + \frac{1}{12 \, a_3} \left( \frac{1}{1944 \, a_2^2 \, a_3^3} + 216 \, a_2^2 \, a_3 \right)
\ln(1 + 12 \, a_0 \, a_3 \, t) \\ \label{solr22}
&& + a_4 \, x + a_5 \, y + a_6,
\end{eqnarray}
and it fulfils
\begin{equation}\nonumber
\lim_{t \nearrow T^*} u(x,y,t)= +\infty
\end{equation}
on the set
\begin{eqnarray}\label{bum1}
&& \left\{ (x,y) \in \mathbb{R}^2 \, \left| \,\,
\frac{1}{46656 \, a_2^2 \, a_3^3} \, x^4
- a_2 \, x \, y^3 \right. \right.
\\ \nonumber && \qquad \qquad \left.
+ \frac{1}{648 \, a_2 \, a_3^2} \, x^3 \, y + 9 \, a_2^4 \, a_3 \, y^4
= 0 \right\}
\end{eqnarray}
and, if $a_0>0$, also on the set
\begin{eqnarray}\label{bum2}
&& \left\{ (x,y) \in \mathbb{R}^2 \, \left| \,\,
\frac{1}{46656 \, a_2^2 \, a_3^3} \, x^4
- a_2 \, x \, y^3 \right. \right.
\\ \nonumber && \qquad \qquad \left.
+ \frac{1}{648 \, a_2 \, a_3^2} \, x^3 \, y + 9 \, a_2^4 \, a_3 \, y^4
< 0 \right\};
\end{eqnarray}
conversely, if $a_0<0$, also on the set
\begin{eqnarray}\label{bum3}
&& \left\{ (x,y) \in \mathbb{R}^2 \, \left| \,\,
\frac{1}{46656 \, a_2^2 \, a_3^3} \, x^4
- a_2 \, x \, y^3 \right. \right.
\\ \nonumber && \qquad \qquad \left.
+ \frac{1}{648 \, a_2 \, a_3^2} \, x^3 \, y + 9 \, a_2^4 \, a_3 \, y^4
> 0 \right\}.
\end{eqnarray}
On the other hand we find
\begin{equation}\nonumber
\lim_{t \nearrow T^*} u(x,y,t)= -\infty
\end{equation}
if $a_0>0$ on the set~\eqref{bum3}, and if $a_0<0$ on the set~\eqref{bum2}.
The solution~\eqref{solr21} blows up in infinite time if $a_0 \, a_3 > 0$ and~\eqref{algebraic} is fulfilled;
so in particular formula~\eqref{solr22} is a valid example of this fact. In this case
\begin{equation}\nonumber
\lim_{t \nearrow \infty} u(x,y,t)= +\infty,
\end{equation}
for all $(x,y) \in \mathbb{R}^2$.
The solution is global in $[0,\infty]$ if $a_0=0$; in such a case it reduces to
\begin{eqnarray}\nonumber
u(x,y,t) = a_4 \, x + a_5 \, y + a_6.
\end{eqnarray}

In the radial case we find
$$
u(r,t)= \frac{a_0}{1 - 48 \, a_0 \, a_1 \, t} \,\, \frac{r^4}{48 \, a_1} + \frac{1}{36 \, a_1^2} \ln(1 - 48 \, a_0 \, a_1 \, t) + a_2,
$$
where $a_0$, $a_1$ and $a_2$ are arbitrary real parameters except that $a_1 \neq 0$.
If $a_0 = 0$ then $u(r,t)=a_2$ and therefore the solution is global in time.
If $a_0 \, a_1 <0$ the solution blows up in infinite time in the following fashion:
\begin{equation}\nonumber
\lim_{t \nearrow \infty} u(r,t)= +\infty,
\end{equation}
for all $r \in [0,\infty[$.
If $a_0 \, a_1 >0$ the solution blows up in finite time:
$$
T^* = \frac{1}{48 \, a_0 \, a_1}.
$$
In this case the blow-up is again complete and such that
\begin{equation}\nonumber
\lim_{t \nearrow T^*} u(0,t)= -\infty,
\end{equation}
and
\begin{equation}\nonumber
\lim_{t \nearrow T^*} u(r,t)= +\infty,
\end{equation}
for all $r \in \, ]0,\infty[$.

We have yet another family of non-radial solutions that might blow up in infinite time
\begin{eqnarray}\nonumber
u(x,y,t) &=& a_1 \, x^2 + a_2 \, y^2 + 4 \, a_1 \, a_2 \, t + a_0 \, \exp[a_3 \, x + (2 \, a_2 \, a_3^2 -a_3^4) \, t] \\ \nonumber
& & + \, a_4 \, x + a_5 \, y + a_6,
\end{eqnarray}
where $a_0$, $a_1$, $a_2$, $a_3$, $a_4$, $a_5$ and $a_6$ are arbitrary real parameters. The solution does not blow up in infinite time
if simultaneously either $a_1=0$ or $a_2=0$ and either $2 \, a_2 \le a_3^2$ or $a_3=0$;
it neither blows up in infinite time if simultaneously either $a_1=0$ or $a_2=0$ and $a_0=0$.
Conversely it does blow up in infinite time if either
$a_1 \, a_2 \neq 0$ or $2 \, a_2 > a_3^2$ and $a_3 \neq 0$ and $a_0 \neq 0$. For this solution we can identify the following types of behavior: if either
$a_3=0$ or $2 \, a_2 \le a_3^2$ or $a_0 = 0$ then
\begin{equation}\nonumber
\lim_{t \nearrow \infty} u(x,y,t)= +\infty,
\end{equation}
for $a_1 \, a_2 > 0$ and all $(x,y) \in \mathbb{R}^2$, and
\begin{equation}\nonumber
\lim_{t \nearrow \infty} u(x,y,t)= -\infty,
\end{equation}
for $a_1 \, a_2 < 0$ and all $(x,y) \in \mathbb{R}^2$. If $a_3 \neq 0$, $a_0 \neq 0$, and $2 \, a_2 > a_3^2$ then
\begin{equation}\nonumber
\lim_{t \nearrow \infty} u(x,y,t)= +\infty,
\end{equation}
if $a_0 > 0$ and for all $(x,y) \in \mathbb{R}^2$, and
\begin{equation}\nonumber
\lim_{t \nearrow \infty} u(x,y,t)= -\infty,
\end{equation}
if $a_0 < 0$ and for all $(x,y) \in \mathbb{R}^2$.

\section{Refinement of a previous blow-up result}
\label{refine}

The aim of this section is to prove that, under exactly the same conditions considered in~\cite{n2}, the $H^2(\Omega)$ blow-up proven there is in fact
a $W^{1,\infty}(\Omega)$ blow-up. We will show that this fact follows as a consequence of the Gagliardo-Nirenberg interpolation.

Along this section we will consider $C$ a universal positive constant which numerical value may change from line to line or even within the same line,
and $\Omega \subset \mathbb{R}^2$ a bounded and open domain provided with a smooth boundary. The setting is exactly the same as the one assumed
in~\cite{n2}.

\begin{lemma}\label{w1infty}
Let $u \in C([0,T],H^2_0(\Omega))$ be a weak solution to the initial-Dirichlet problem for equation~\eqref{pareq} as defined in~\cite{n2}
such that
$$
\int_0^T \| \nabla u \|_{L^\infty(\Omega)}^4 < \infty \,.
$$
Then
$$
\sup_{0 \le s \le t} \| u \|_{H^2(\Omega)}^2 + \int_0^t \| u \|_{H^4(\Omega)}^2 \le C_1 \,
\exp \left[ C_2 \int_0^t \| \nabla u \|_{L^\infty(\Omega)}^4 \right] \| u_0 \|_{H^2(\Omega)}^{2} \,,
$$
for all $0 \le t \le T$ and some positive universal constants $C_1, C_2$, where the initial condition $u_0 \in H^2_0(\Omega)$.
\end{lemma}

\begin{proof}
Test equation~\eqref{pareq} against $\Delta^2 u$ and integrate by parts to find
\begin{eqnarray}\nonumber
\frac12 \frac{d}{dt} \| \Delta u \|_{L^2(\Omega)}^2 + \| \Delta^2 u \|_{L^2(\Omega)}^2 \!\! &=& \!\! \int_\Omega \Delta^2 u \det(D^2 u) \\ \nonumber
&\le& \!\! C \|\Delta u\|_{L^4(\Omega)}^2 \|\Delta^2 u\|_{L^2(\Omega)},
\end{eqnarray}
after integration by parts and the use of H\"older inequality.
Now compute
\begin{eqnarray}\nonumber
\|\Delta u\|_{L^4(\Omega)}^4 &=& \int_\Omega |\Delta u|^4 \\ \nonumber
&=& -\int_\Omega \nabla u \cdot \nabla (\Delta u)^3 \\ \nonumber
&\le& C \|\nabla u\|_{L^\infty(\Omega)} \|\Delta \nabla u\|_{L^2(\Omega)} \|\Delta u\|_{L^4(\Omega)}^2,
\end{eqnarray}
after integrating by parts and using H\"older inequality.
Consider the Gagliardo-Nirenberg interpolation
$$
\| \Delta \nabla u \|_{L^2(\Omega)} \le C \| \Delta u \|_{L^2(\Omega)}^{1/2} \| \Delta^2 u \|_{L^2(\Omega)}^{1/2} \,;
$$
Rearranging all the estimates we find
$$
\|\Delta u\|_{L^4(\Omega)}^2 \le C \|\nabla u\|_{L^\infty(\Omega)} \|\Delta \nabla u\|_{L^2(\Omega)},
$$
and
\begin{eqnarray}\nonumber
\frac12 \frac{d}{dt} \| \Delta u \|_{L^2(\Omega)}^2 + \| \Delta^2 u \|_{L^2(\Omega)}^2 \!\! &\le& \!\!
C \|\nabla u\|_{L^\infty(\Omega)} \| \Delta u \|_{L^2(\Omega)}^{1/2} \| \Delta^2 u \|_{L^2(\Omega)}^{3/2} \\ \nonumber
&\le& \!\! C \|\nabla u\|_{L^\infty(\Omega)}^4 \| \Delta u \|_{L^2(\Omega)}^{2} +\frac12 \| \Delta^2 u \|_{L^2(\Omega)}^{2},
\end{eqnarray}
after the use of Young inequality.
In consequence
$$
\frac{d}{dt} \| \Delta u \|_{L^2(\Omega)}^2 \le C \|\nabla u\|_{L^\infty(\Omega)}^4 \| \Delta u \|_{L^2(\Omega)}^{2},
$$
and thus Gr\"onwall inequality leads to
$$
\| \Delta u \|_{L^2(\Omega)}^2 \le \exp \left[ C \int_0^t \| \nabla u \|_{L^\infty(\Omega)}^4 \right] \| \Delta u_0 \|_{L^2(\Omega)}^{2}
\quad \forall \, 0 \le t \le T \,.
$$
The last three equations lead to
$$
\max_{0 \le s \le t} \| \Delta u \|_{L^2(\Omega)}^2 + \int_0^t \| \Delta^2 u \|_{L^2(\Omega)}^2 \le 2 \,
\exp \left[ C \int_0^t \| \nabla u \|_{L^\infty(\Omega)}^4 \right] \| \Delta u_0 \|_{L^2(\Omega)}^{2} \,,
$$
$\forall \, 0 \le t \le T$.
\end{proof}

\begin{corollary}
There exists a non-empty $\mathcal{U} \subset H^2_0(\Omega)$ such that if $u_0 \in \mathcal{U}$ then there exists a $T^* < \infty$
such that
$$
\int_0^{T^*} \| \nabla u \|_{L^\infty(\Omega)}^4= \infty.
$$
Moreover we have $\limsup_{t \nearrow T^*} \|u\|_{H^2(\Omega)}=\infty$ if and only if the above equality holds.
\end{corollary}

\begin{proof}
This result is a direct consequence of Lemma~\ref{w1infty} and the results in~\cite{n2}.
\end{proof}

\begin{remark}
Note that the weak solution to the initial-Dirichlet problem for equation~\eqref{pareq} ceases to exist in a finite time $T^*$ if and only if
$\limsup_{t \nearrow T^*} \|u\|_{H^2(\Omega)}=\infty$, see~\cite{n2}.
\end{remark}

\section{Conclusions and open questions}
\label{conclusions}

In section~\ref{explicit} we have found explicit solutions to equation~\eqref{pareq} in the unit square, the unit disc, and in the plane.
We have exhaustively analyzed the blow-up structure of all of these solutions except for~\eqref{solr21}, which is the more complex by far;
anyway, in this case partial results are discussed. For all the other solutions complete blow-up is always found, with a striking regularity in
its structure. If the blow-up is in infinite time then all points go to either $+\infty$ or $-\infty$, and both possibilities do happen.
If the blow-up happens in finite time, then some points go to $+\infty$ and others to $-\infty$. We have not found counterexamples to these two
facts even in the behavior of~\eqref{solr21}; for instance, in the case of infinite time blow-up exhaustively analyzed for this solution we have
again found that these rules are fulfilled. Therefore, as open questions we may list the determination of the conditions under which the blow-up
is necessarily complete and when there are regions of the plane which do not blow up. Also, if it is possible to find solutions for which the
infinite time blow-up can happen simultaneously to $+\infty$ and $-\infty$; conversely, if the finite time blow-up may happen only towards $+\infty$
or $-\infty$, but not to both. Some answers to these questions may arise if solution~\eqref{solr21} is further understood or if new
explicit formulas are derived.

To fully understand the blow-up structure of solution~\eqref{solr21} one needs to understand in turn the sets
\begin{eqnarray}\nonumber
&& \left\{ (x,y) \in \mathbb{R}^2 \left| \,\,
\frac{3456 \, a_1^3 \, a_3^3 \, + 432 \, a_1^2 \, a_3^2 -1}{46656 \, a_2^2 \, a_3^3} \, x^4 \right. \right.
+ a_1 \, x^2 \, y^2 + a_2 \, x \, y^3 \\ \nonumber && \hspace{2.25cm}
+ \frac{288 \, a_1^3 \, a_3^2 + 12 \, a_1 \, a_3 -1}{648 \, a_2 \, a_3^2} \, x^3 \, y + \left. \frac{9 \, a_2^4 \, a_3}{24 \, a_1 \, a_3 -1} \, y^4
= 0 \right\},
\end{eqnarray}
\begin{eqnarray}\nonumber
&& \left\{ (x,y) \in \mathbb{R}^2 \left| \,\,
\frac{3456 \, a_1^3 \, a_3^3 \, + 432 \, a_1^2 \, a_3^2 -1}{46656 \, a_2^2 \, a_3^3} \, x^4 \right. \right.
+ a_1 \, x^2 \, y^2 + a_2 \, x \, y^3 \\ \nonumber && \hspace{2.25cm}
+ \frac{288 \, a_1^3 \, a_3^2 + 12 \, a_1 \, a_3 -1}{648 \, a_2 \, a_3^2} \, x^3 \, y + \left. \frac{9 \, a_2^4 \, a_3}{24 \, a_1 \, a_3 -1} \, y^4
< 0 \right\},
\end{eqnarray}
and
\begin{eqnarray}\nonumber
&& \left\{ (x,y) \in \mathbb{R}^2 \left| \,\,
\frac{3456 \, a_1^3 \, a_3^3 \, + 432 \, a_1^2 \, a_3^2 -1}{46656 \, a_2^2 \, a_3^3} \, x^4 \right. \right.
+ a_1 \, x^2 \, y^2 + a_2 \, x \, y^3 \\ \nonumber && \hspace{2.25cm}
+ \frac{288 \, a_1^3 \, a_3^2 + 12 \, a_1 \, a_3 -1}{648 \, a_2 \, a_3^2} \, x^3 \, y + \left. \frac{9 \, a_2^4 \, a_3}{24 \, a_1 \, a_3 -1} \, y^4
> 0 \right\}.
\end{eqnarray}
These make up the set~\eqref{bumanifold} and, when $a_1=0$, give rise to the sets~\eqref{bum1}, \eqref{bum2}, and~\eqref{bum3}.
In which of these sets we are located, along with the sign of $a_0$, will determine if the blow up is towards $+\infty$, $-\infty$, or if there
are regions of the plane in which the blow-up does not happen. In this respect it is also important to understand the condition
\begin{equation}\nonumber
4 \, a_1 +
\frac{3456 \, a_1^3 \, a_3^3 + 432 \, a_1^2 \, a_3^2 -1}{1944 \, a_2^2 \, a_3^3} + \frac{216 \, a_2^2 \, a_3}{24 \, a_1 \, a_3 -1}
= 0,
\end{equation}
complementary of~\eqref{algebraic}, which may be relevant for finding regions of the plane which do not blow up.

Also, theoretical developments in~\cite{n2,xu,xuzhou,zhou,zhou2}, and in section~\ref{refine} too, show {\it a priori} estimates that
suggest a less dramatic blow-up than the one found in the explicit examples presented in section~\ref{explicit}. It remains to be
clarified if this theoretical framework could be refined in order to capture the real nature of the blowing up solutions, or if on the contrary
the boundary conditions play such a fundamental role that they forbid the wild divergences found herein in the cases studied in those
references. On the other hand, the infinite time blow-up for the initial-Dirichlet problem was ruled out in~\cite{n2}, but it is present
in some of these explicit solutions. So the conditions that allow this type of singular behavior are still to be revealed, as well as their
potential connection with the other features that characterize the singular solutions we have herein introduced.

\section*{Acknowledgements}

The author gratefully acknowledges Filippo Gazzola and Michael Winkler for discussions.
This work has been partially supported by the Government of Spain (Ministerio de Ciencia, Innovaci\'on y Universidades)
through Project PGC2018-097704-B-I00.

\vskip5mm
\noindent
{\footnotesize
Carlos Escudero\par\noindent
Departamento de Matem\'aticas Fundamentales\par\noindent
Universidad Nacional de Educaci\'on a Distancia\par\noindent
{\tt cescudero@mat.uned.es}\par\vskip1mm\noindent
}

\end{document}